\documentclass[12pt]{amsart}
\usepackage{amsmath,amssymb,amsfonts}
\usepackage{enumerate}
\usepackage{anysize}
\marginsize{3cm}{3cm}{3cm}{3cm}
\input xy
\xyoption{all}
\usepackage{pb-diagram}
\usepackage[all]{xy}
\input xy
\xyoption{all}
\theoremstyle{plain}
\newtheorem{lemma}{Lemma}

\newtheorem{corollary}{Corollary}

\newtheorem{proposition}{Proposition}

\newcommand{\ST}{\mathbb{S}}

\newcommand{\C}{\mathbb{C}}

\newcommand{\End}{{\rm End}\,}

\def\St{{\rm St}}

\def\SU{{\rm SU}}

\def\GL{{\rm GL}}

\def\Sym{{\rm Sym}}
\def\sym{{\rm Sym}}

\begin{document}

\title[A Cauchy-Schwarz inequality for  representations of $\SU(2)$]
{A Cauchy-Schwarz inequality for \\ representations of $\SU(2)$}

\author{Dipendra Prasad}

\address{School of Mathematics, Tata Institute of Fundamental
Research, Colaba, Mumbai-400005, INDIA}
\email{dprasad@math.tifr.res.in}

\begin{abstract} 

Motivated by some applications to calculating order of poles of certain (local or global) $L$-functions, the author considers a 
Cauchy-Schwarz type inequality for representations of $\SU(2)$.

\end{abstract}

\maketitle

\section*{\bf 1.}
For any 
finite dimensional representation $E$ of $\SU(2)$, 
let $E[i]$ denote the 
subspace of  $E$ on which $\ST^1$, the standard maximal torus in $\SU(2)$, operates by the character
$z\rightarrow z^i.$

Let $V$ and $W$ be two finite dimensional representations of $\SU_2$.
This note tries to  estimate  
$$\dim (V \otimes V[0] \oplus W \otimes W[0]) - 2\dim (V \otimes W[1]),$$
or, a closely related function:
$$d(V,W)= \dim (V \otimes V[2] \oplus W \otimes W[2]) - 2\dim (V \otimes W[1]).$$ Since the space of finite 
dimensional (virtual) representations of $\SU(2)$ is 
itself a free abelian group  with a natural basis consisting of the 
isomorphism classes of irreducible representations of $\SU(2)$, the question
about $d(V,W)$ amounts to one regarding:
$$d(V,W) = q(V) + q(W)-2B(V,W),$$
but notice that in the above expression for $d(V,W)$, the quadratic form $q$ and the bilinear form $B$ are different, and as a result there seems  no natural 
context in which $d(V,W)$ may be studied. Since $q$ is a positive definite quadratic form, and 
$B$ is closely related to its associated bilinear form, $d(V,W)$  
is typically non-negative, say for $\dim V - \dim W=1$, but not always; one of the motivating question has been if we can classify the pairs $(V,W)$ for which $d(V,W)$ is negative, or with $d(V,W)=0$?  We attempt some of this here, but without having any road-map, we seem rather lost. 

For any 
finite dimensional representation $E$ of $\SU(2)$, write $E = E^+ \oplus E^-$ as the sum of the subspaces on which the center
of $\SU(2)$ operates trivially, and non-trivially;  thus for representations $V, W$, write,
\begin{eqnarray*}
V = V^+ \oplus V^-,\\
W= W^+ \oplus W^-.
\end{eqnarray*}
Since,
\begin{eqnarray*}
V \otimes V [0] \oplus W \otimes W [0] & = & V^+ \otimes V^+[0] \oplus V^- \otimes V^-[0] \oplus W^+\otimes W^+ [0] \oplus  W^- \otimes W^-[0] \\
V \otimes W [1] & = & V^+ \otimes W^- [1] \oplus V^- \otimes W^+[1],
\end{eqnarray*}
it follows that,
\begin{eqnarray*}
V \otimes V [0] + W \otimes W [0] -2V \otimes W[1] & = & V^+ \otimes V^+ [0] + W^- \otimes W^- [0] -2V^+ \otimes W^-[1] \\
& + &  V^- \otimes V^- [0] + W^+ \otimes W^+ [0] -2V^- \otimes W^+[1].
\end{eqnarray*}

Therefore to understand the behavior of the function 
$$\dim (V \otimes V[0] \oplus W \otimes W[0]) - 2 \dim (V \otimes W[1]),$$ it suffices to assume
that the center of $\SU(2)$ operates by scalar on $V$ and on $W$, and these scalars are different.

\begin{lemma}
Let $V$ and $W$ be two finite dimensional representations of $\SU_2$ such that the center of $\SU(2)$ operates 
by scalars on $V$ and on $W$.  Then we have the estimates,
\begin{enumerate}
\item  If the central characters of $V$ and $W$ are the same, $$\dim (V \otimes V[0] \oplus W \otimes W[0])  \geq 2 \dim (V \otimes W[0]);$$
further, the inequality becomes an equality if and only if $V=W$. (This part of the lemma remains true for the 
zero weight space for a maximal torus in any compact Lie group.)
\item  If the central characters of $V$ and $W$ are opposite, 
$$\dim (V \otimes V[0] \oplus W \otimes W[0])  \geq 2 \dim (V \otimes W[1]);$$
further, the inequality is never an equality.
\end{enumerate}
\end{lemma}

\begin{proof}
Let $\chi_1$ (resp. $\chi_2$) denote the character of the representation $V$ (resp. $W$).
The first part of the lemma follows by integrating on $\ST^1$ the inequality, 
$$(\chi_1-\chi_2)^2 = \chi_1^2+\chi_2^2-2 \chi_1 \chi_2 \geq 0.$$ 

For the proof of the second part of the lemma, 
let $\St_i$ denote the character of the standard $(i+1)$-dimensional irreducible representation of $\SU_2$, i.e. $\Sym^i(\C^2)$, which is a  real valued 
function on $\SU_2$. By the Clebsch-Gordon theorem, $\St_i^2$ is the character of 
the sum of all irreducible representations of $\SU_2$ with multiplicity one  of highest weights $\leq 2i$, and of trivial central character, and 
similarly $\St_i \cdot \St_{i-1}$ is the character of 
the sum of all irreducible representations of $\SU_2$ with multiplicity one 
of highest weights $\leq 2i-1$, and of non-trivial central character.
 It follows from the
representation theory of $\SU_2$ that any irreducible representation of $\SU_2$ with trivial central 
character contains the trivial character of $\ST^1$ with multiplicity one, whereas any irreducible representation of $\SU_2$ with non-trivial central 
character contains the character $z\rightarrow z$ of $\ST^1$ with multiplicity one. Integrating the following 
inequality for $n$ sufficiently large (this time on $\SU_2$ on which we will fix a Haar measure giving it volume 1)
$$(\chi_1 \St_n-\chi_2 \St_{n-1})^2 
= \chi_1^2 \St_n^2+\chi_2^2 \St_{n-1}^2 -2 \chi_1 \chi_2 \St_n \St_{n-1}\geq 0,$$ 
and interpreting the various terms using the Schur orthogonality, and the facts on representation theory of $\SU_2$ 
just recalled, the inequality in the second part of the lemma follows. 

To see that the inequality in the second part of the Lemma can never be an equality, note that the equality can happen only if 
$\chi_1 \St_n-\chi_2 \St_{n-1}=0$. But by the same argument, we must then also have, $\chi_1 \St_{n-1}-\chi_2 \St_{n} = 0$. These two identities 
mean that $\chi_1 = \pm \chi_2$,  which is not an option because of different central characters. 
\end{proof}

\begin{corollary}
Let $V =\sum_{i=0}^{i=d} n_i \sym^{2i}(\C^2)$, and 
$W  = \sum_{i=1}^{i=d} m_i \sym^{2i-1}(\C^2)$ 
 be two finite dimensional representations of $\SU_2$.
For any 
representation $E$ of $\SU(2)$, let $E[i]$ denote the 
subspace of  $E$ on which $\ST^1$, the standard maximal torus in $\SU(2)$, operates by the character
$z\rightarrow z^i.$   Then we have the estimate,
\begin{eqnarray*} 2 \dim (V \otimes W[1]) - \dim (V \otimes V[2] \oplus W \otimes W[2]) 
& \leq & \sum n_i^2 + \sum m_i^2 .
\end{eqnarray*}
\end{corollary}
\begin{proof} The proof is an immediate consequence of the previous lemma on noting the well-known fact that for
any representation $\pi$ of $\SU_2$ with trivial central character,
$$\dim \pi[0] - \dim \pi[2] = \dim \pi^{\SU_2},$$
and therefore for representations $V,W$ as above,
\begin{eqnarray*}
V \otimes V [0] - V \otimes V [2] = \dim \End _{\SU_2}[V] & = & \sum_{i=0}^{i=d} n^2_i \\
W \otimes W [0] - W \otimes W[2] = \dim \End_{\SU_2}[W] & = & \sum_{i=1}^{i=d} m^2_i.
\end{eqnarray*} 
\end{proof}

To get some feel for the dimensions involved, let's do an example. Let
\begin{eqnarray*}
V & = & n_0 \cdot 1  \oplus n_2 \sym^{2}(\C^2) \oplus n_{2d} \sym^{2d}(\C^2), \\
W & = & m_1 \sym^1(\C^2) \oplus m_{2r-1}\sym^{2r-1}(\C^2),
\end{eqnarray*}
assuming $2d > 2r-1$. Then,
\begin{eqnarray*}
(V\otimes V)[0] & = & n_0^2 + 3n_2^2 +(2d+1)n_{2d}^2 + 2n_0n_2+2n_0n_{2d} +6n_2n_{2d}, \\
(W \otimes W)[0]& = & 2m_1^2 + 2rm_{2r-1}^2 + 4m_1m_{2r-1},\\
(V \otimes W)[1] & = & n_0m_1+n_0m_{2r-1} + 2n_2m_1+3n_2m_{2r-1} +2n_{2d}m_1+2rn_{2d}m_{2r-1}.
\end{eqnarray*}

Further, 
\begin{eqnarray*}
\dim (\End V) & = & n_0^2 + n_2^2 +n_{2d}^2 \\
\dim (\End W) & = & m_1^2 + m_{2r-1}^2 .
\end{eqnarray*}

Therefore,
\begin{eqnarray*}
& & \dim (V \otimes V[0] \oplus W \otimes W[0]) - 2 \dim (V \otimes W[1])\\
& = & n_0^2 + 3n_2^2 +(2d+1)n_{2d}^2 + 2n_0n_2+2n_0n_{2d} +6n_2n_{2d} + 2m_1^2 + 2rm_{2r-1}^2 + 4m_1m_{2r-1}\\
&-& 2(n_0m_1+n_0m_{2r-1} + 2n_2m_1+3n_2m_{2r-1} +2n_{2d}m_1+2rn_{2d}m_{2r-1}),
\end{eqnarray*}
which by Lemma 1 is always positive. On the other hand,
\begin{eqnarray*}
& & \dim (V \otimes V[2] \oplus W \otimes W[2]) -  2 \dim (V \otimes W[1])\\
& = &  2n_2^2 +2d n_{2d}^2 + 2n_0n_2+2n_0n_{2d} +6n_2n_{2d} + 2m_1^2 + 2rm_{2r-1}^2 + 4m_1m_{2r-1}\\
&-& 2(n_0m_1+n_0m_{2r-1} + 2n_2m_1+3n_2m_{2r-1} +2n_{2d}m_1+2rn_{2d}m_{2r-1}),
\end{eqnarray*}
is `usually' positive but not always, and something which we are not able fully understand. We may want to impose an 
auxiliary condition such as $\dim V - \dim W =1$.

\section*{\bf 2.}

Rest of this paper is written with the aim of improving upon the inequality offered by Lemma 1
 in the case of different central characters, or equivalently to understand given the representation $V$ of $\SU(2)$, what
are the choices for the representation $W$ so that the difference of the left hand side of the inequality in Lemma 1 with the right hand side is as small as possible;
we do this by analyzing the argument 
which goes into the proof of the lemma more carefully.

Write,
\begin{eqnarray*}
V & = & \sum_{i=0}^{i=d} n_i \sym^{2i}(\C^2), \\
W & = & \sum_{i=1}^{i=d} m_i \sym^{2i-1}(\C^2).
\end{eqnarray*}
We use Clebsch-Gordon theorem to expand the character $\chi_1 \St_{n} - \chi_2 \St_{n+1}$ out:

\begin{eqnarray*}
\chi_1 \St_{n} - \chi_2 \St_{n+1} 
& = & \sum_{i=0}^{i=d} n_i \St_{2i} \St_{n} - \sum_{i=1}^{i=d}m_i\St_{2i-1} \St_{n+1} \\
& = & \sum_{i=0}^{i=d} n_i ( \St_{n+2i} + \cdots + \St_{n-2i}) - \sum_{i=1}^{i=d}m_i(\St_{n+2i}  + \cdots + 
\St_{n-2i+2}) \\
& = & n_0\St_{n}+ \sum_{i=1}^{i=d}(n_i-m_i)(\St_{n+2i} + \cdots + \St_{n-2i}) + \sum_{i=1}^{i=d} m_i  \St_{n-2i}   \\
& \stackrel{\star}{=} &  (n_d-m_d)\St_{n+2d} + (n_d-m_d +n_{d-1}-m_{d-1})\St_{n+2d-2} + \cdots \\ & & + 
(n_d-m_d +n_{d-1}-m_{d-1}+n_{d-2} + \cdots + n_1-m_1)\St_{n+2}  \\ & & + 
(n_d-m_d +n_{d-1}-m_{d-1}+n_{d-2} + \cdots + n_1-m_1 +n_0)\St_{n}  \\ & & + 
(n_d-m_d +n_{d-1}-m_{d-1}+n_{d-2} + \cdots + n_1-m_1 + m_1)\St_{n-2} \\ & &  + \cdots +  
(n_d-m_d + m_d)\St_{n-2d}.
\end{eqnarray*}

Using the inequality,
$$a^2 + b^2 \geq \frac{1}{2} (a-b)^2,$$
and grouping the first term in the above summation with the last, then the second with the second last etc.
(which leaves the middle term alone, which we do not consider), we get:
$$\int _{\SU_2}[\chi_1 \St_{n} - \chi_2 \St_{n-1}]^2 \geq \frac{1}{2}(m_1^2 + \cdots + m_d^2).$$

Similarly, using the same inequality,
$$a^2 + b^2 \geq \frac{1}{2} (a-b)^2,$$
but now grouping the first term in the equation $(\star)$  with the second last term, then the 
second with the third last etc., we get:
$$\int _{\SU_2}[\chi_1 \St_{n} - \chi_2 \St_{n-1}]^2 \geq \frac{1}{2}(n_0^2 +n_1^2 + \cdots + n_{d-1}^2+2n_d^2) 
.$$

We record this in the following lemma improving Lemma 1.

\begin{lemma}
Let $V =\sum_{i=0}^{i=d} n_i \sym^{2i}(\C^2)$, and 
$W  = \sum_{i=1}^{i=d} m_i \sym^{2i-1}(\C^2)$ 
 be two finite dimensional representations of $\SU_2$.
Assume that the character of $V$ is $\chi_1$ 
and that of $W$ is $\chi_2$. Then for any large enough integer $n$, we have the estimate,

$$\int _{\SU_2}[\chi_1 \St_{n} - \chi_2 \St_{n-1}]^2  \geq  \frac{1}{2}\max\{(m_1^2 + \cdots + m_d^2), (n_0^2 +n_1^2 + \cdots + n_{d-1}^2+ 2n_d^2)\},  
 $$
and hence, 
$$\dim (V \otimes V[0] \oplus W \otimes W[0])- 2 \dim (V \otimes W[1])  
\geq  \frac{1}{2} \max \{(m_1^2 + \cdots + m_d^2), (n_0^2 +n_1^2 + \cdots + 2n_d^2)\} .$$
\end{lemma}

\section*{\bf 3.}

We analyze a little further the  
expression above  for $\chi_1\St_{n} - \chi_2 \St_{n+1} $ in equation $(\star)$ in the hope of finding an inequality in the other direction.
We have, 
$$\chi_1\St_{n} - \chi_2 \St_{n+1} = a_d\St_{n+2d} + \cdots + a_1 \St_{n+2} + (a_1+n_0) \St_{n} + (a_1+m_1)\St_{n-2} 
+ \cdots + (a_d+m_d)\St_{n-2d},$$
where 
$$a_i = \sum _{j \geq i} (n_j-m_j).$$

Therefore, 
\begin{eqnarray*} 
\int _{\SU_2}[\chi_1 \St_{n} - \chi_2 \St_{n+1}]^2  
& = & (a_1+n_0)^2+  \sum_{i\geq 1}{[a_i^2 + (a_i+m_i)^2]} \\ & = & 
(a_1+n_0)^2+ \sum_{i \geq 1} 2a_i(a_i+m_i) + 
\sum_{i \geq 1} m_i^2.
\end{eqnarray*}

It follows that the $L^2$-norm of $\chi_1 \St_{n} - \chi_2 \St_{n+1}$ is controlled by the signs of $a_i(a_i+m_i)$, and among all cases,
a favorable case would be when all the $a_i(a_i+m_i)$ have the same sign. This does happen in an interesting set of cases which is what we analyse
next. 

For a representation,
$$E =\sum_{i=0}^{i=d} \ell_i \sym^{i}(\C^2),$$
let
$E^-$ denote the representation,
$$E^- =\sum_{i=0}^{i=d} \ell_i \sym^{i-1}(\C^2),$$
with the understanding that $\sym^{-1}(\C^2) = 0$.

Write the representations $V =\sum_{i=0}^{i=d} n_i \sym^{2i}(\C^2)$, and 
$W  = \sum_{i=1}^{i=d} m_i \sym^{2i-1}(\C^2)$, in terms of partitions as:

$$V= \underbrace{2d+1,\cdots, 2d+1}_{n_d}, \underbrace{2d-1,\cdots, 2d-1}_{n_{d-1}},  \cdots, \underbrace{3,\cdots, 3}_{n_1}, 
\underbrace{1,\cdots, 1}_{n_0}, 
$$
and
$$W= \underbrace{2d,\cdots, 2d}_{m_d}, \underbrace{2d-2,\cdots, 2d-2}_{m_{d-1}},  \cdots, \underbrace{2,\cdots, 2}_{m_1}.$$
The corresponding partitions associated to $V^-$ and $W^-$ are:
$$V^-= \underbrace{2d,\cdots, 2d}_{n_d}, \underbrace{2d-2,\cdots, 2d-2}_{n_{d-1}},  \cdots, \underbrace{2,\cdots, 2}_{n_1},$$
and
$$W^-= \underbrace{2d-1,\cdots, 2d-1}_{m_d}, \underbrace{2d-3,\cdots, 2d-3}_{m_{d-1}},  \cdots, \underbrace{1,\cdots, 1}_{m_1}.$$

On the set of partitions $\underline{a} = a_1m +  a_2(m-1) + a_3(m-1) +\cdots + a_m1$, 
there is a partial ordering defined
by $\underline{a} \geq_Y \underline{b}$  
for $\underline{b} = b_1m +  b_2(m-1) + b_3(m-1) +\cdots + b_m1$, 
if 
 $a_1+\cdots +a_t \geq b_1+\cdots + b_t$ for all integers $t$. This partial order $\geq_Y$ on the set of all 
partitions is called the {\it Young order} giving rise to  {\it Young lattice}. It is different from the 
much more used partial order on the set of all partitions which is called the {\it Dominance} order. The Young order has the property that 
$\underline{a} \geq_Y \underline{b}$  
if and only if if we were to represent $\underline{a}$ and  $\underline{b}$  
as Young tableaux, then $\underline{b}$ sits inside $\underline{a}$ as a sub-tableaux, i.e., 
$\underline{b}$ is obtained from  $\underline{a}$ by removing some boxes from the end of each row.

An important property of Young order is that  if $\underline{a}$ is a partition of $n+1$ 
corresponding to an irreducible representation $\pi_{\underline{a}}$ of $S_{n+1}$, and 
 $\underline{b}$ is a partition of $n$ corresponding to an irreducible representation $\pi_{\underline{b}}$ of $S_{n}$, 
then  the representation $\pi_{\underline{a}}$ of $S_{n+1}$ contains the representation  $\pi_{\underline{b}}$ of $S_{n}$ 
when restricted to $S_n$ if and only if   
$\underline{a} \geq_Y \underline{b}$. (Although this note is motivated by questions on branching laws from $\GL({n+1})$ to $\GL(n)$ in the p-adic context, it is not clear if Young order, and its relation to branching law from $S_{n+1}$ to $S_n$ is really relevant to this note, or it is a red herring.)

Therefore, the condition $W \geq_Y V^-$ means,
\begin{eqnarray*}
m_d &\geq & n_d, \\ 
m_d + m_{d-1}&\geq & n_d + n_{d-1}, \\ 
m_d + m_{d-1} + m_{d-2} & \geq & n_d + n_{d-1} + n_{d-2}, \\ 
\cdots  & \geq  & \cdots, \\ 
m_d + m_{d-1} + m_{d-2}+ \cdots+ m_1 & \geq & n_d + n_{d-1} +  n_{d-2} + \cdots + n_1.
\end{eqnarray*}
These conditions can be reinterpreted in terms of $a_i$ defined earlier as $a_i = \sum _{j \geq i} (n_j-m_j)$ to
simply $a_i \leq 0$ for all $i \geq 1$.
 
The condition $V \geq_Y W^-$ means,
\begin{eqnarray*}
n_d & \geq &  0 \\
n_d + n_{d-1}& \geq & m_d, \\ 
n_d + n_{d-1} + n_{d-2} & \geq & m_d + m_{d-1}, \\ 
n_d + n_{d-1} + n_{d-2} + n_{d-3} & \geq & m_d + m_{d-1} + m_{d-2}, \\ 
\cdots  & \geq  & \cdots, \\ 
n_d + n_{d-1} + n_{d-2}+ \cdots+ n_1 & \geq & m_d + m_{d-1} +  m_{d-2} + \cdots + m_2,
\end{eqnarray*}
which can be rewritten as $a_i+m_i \geq 0$ for all $i \geq 1$.

Thus, it follows that for representations $V,W$ with $W \geq_Y V^-$ and $V \geq_Y W^-$, 
we have,
$$a_i(a_i+m_i) \leq 0.$$
Thus we have proved the following proposition.

\begin{proposition}
Let $V =\sum_{i=0}^{i=d} n_i \sym^{2i}(\C^2)$, and 
$W  = \sum_{i=1}^{i=d} m_i \sym^{2i-1}(\C^2)$ 
 be two finite dimensional representations of $\SU_2$.  Then for
$W \geq_Y V^-$ and $V \geq_Y W^-$, and for $n$ large enough,
 we have the estimate,
$$\int _{\SU_2}[\chi_1 \St_{n} - \chi_2 \St_{n-1}]^2  \leq  (m_1^2 + \cdots + m_d^2) + (a_1+n_0)^2,   
 $$
and therefore,
$$\dim (V \otimes V[0] \oplus W \otimes W[0])- 2 \dim (V \otimes W[1])  
\leq  (m_1^2 + \cdots + m_d^2) + (a_1+n_0)^2 .$$

\end{proposition}

\section*{\bf 4.}

We end the note by discussing  an example.

\vspace{4mm}

\noindent{\bf Example:} Let 
\begin{eqnarray*}
V &=& \sum_{i=0}^d (n_i +m_i) \sym^{i}(\C^2), \\
W & =& \sum_{i=0}^d n_i\sym^{i-1}(\C^2) + \sum_{i=0}^d m_i \sym^{i+1}(\C^2),
\end{eqnarray*}
with character of $V$ to be $\chi_1$ and character of $W$ to be $\chi_2$. 
It can be seen that for these two representations of $\SU_2$, the conditions in 
Proposition 1, i.e.,  
$W \geq V^-$ and $V \geq W^-$, are satisfied.

Since,
\begin{eqnarray*}
\St_i \St_n - \St_{i+1}\St_{n+1}  & = & [\St_{n+i} + \cdots + \St_{n-i}] -  [\St_{n+i+2} + \cdots + \St_{n-i}]   \\
& = & -\St_{n+i+2},
\end{eqnarray*}
and similarly,
\begin{eqnarray*}
\St_i \St_n - \St_{i-1}\St_{n+1}  & = & [\St_{n+i} + \cdots + \St_{n-i}] -  [\St_{n+i} + \cdots + \St_{n-i+2}]   \\
& = & \St_{n-i}.
\end{eqnarray*}
it follows that,
\begin{eqnarray*}
\chi_1 \St_n - \chi_2 \St_{n+1} 
& = & \sum n_i(\St_i \St_n -\St_{i-1}\St_{n+1}) + \sum m_i(\St_i \St_n -\St_{i+1}\St_{n+1}) \\
& =& \sum n_i\St_{n-i} - \sum m_i\St_{n+i+2}.
\end{eqnarray*}

Therefore for $n$ large enough,
\begin{eqnarray*}
\int_{\SU(2)} [\chi_1 \St_n - \chi_2 \St_{n+1}]^2  = \sum(n_i^2 +m_i^2),
\end{eqnarray*}
and therefore,
$$\dim (V \otimes V[0] 
\oplus W \otimes W[0])- 2 \dim (V \otimes W[1]) 
 = \sum(n_i^2 +m_i^2).$$

\vspace{4mm}

\noindent{\bf Remark :} Putting $n_i=m_i$ in this example, we find that 
in a certain sense, Lemma 2 is the best possible.

\end{document}